\pdfoutput=1
\documentclass[11pt,a4paper]{amsart}
\usepackage[margin=1in]{geometry}
\usepackage[utf8]{inputenc}
\usepackage[T1]{fontenc}
\usepackage{textcase}
\usepackage[dvipsnames]{xcolor}
\usepackage{microtype}
\usepackage{fnpct}

\usepackage{amsmath}
\usepackage{amssymb}
\usepackage{eucal}
\usepackage{mathrsfs}
\usepackage{tikz-cd}
\renewcommand{\injlim}{\varinjlim}

\usepackage{enumitem}
\usepackage{booktabs}
\usepackage[pdfusetitle,colorlinks]{hyperref}
\hypersetup{bookmarksdepth=2,pdfencoding=unicode,allcolors=MidnightBlue}
\usepackage[capitalise,noabbrev,nosort]{cleveref}

\crefname{equation}{}{}
\crefformat{enumi}{(#2#1#3)}
\crefname{enumi}{}{}
\newlist{conenum}{enumerate}{1}
\setlist[conenum,1]{label=(\roman*),ref=\roman*}
\creflabelformat{conenumi}{(#2#1#3)}
\crefname{conenumi}{}{}

\numberwithin{equation}{section}
\theoremstyle{plain}
\newtheorem{Theorem}{Theorem}

\crefname{Theorem}{Theorem}{Theorems}
\newtheorem{conjecture}[equation]{Conjecture}
\newtheorem{corollary}[equation]{Corollary}
\newtheorem{lemma}[equation]{Lemma}
\newtheorem{proposition}[equation]{Proposition}
\newtheorem{theorem}[equation]{Theorem}

\theoremstyle{definition}
\newtheorem{definition}[equation]{Definition}
\newtheorem{example}[equation]{Example}

\theoremstyle{remark}

\newtheorem{remark}[equation]{Remark}

\let\oldtop\top\let\top\relax

\newcommand{\NN}{\mathbf{N}}
\newcommand{\ZZ}{\mathbf{Z}}

\newcommand{\RR}{\mathbf{R}}
\newcommand{\CC}{\mathbf{C}}

\newcommand{\E}{\mathrm{E}}

\newcommand{\top}{\textnormal{top}}

\newcommand{\Alg}{\operatorname{Alg}}
\newcommand{\End}{\operatorname{End}}
\newcommand{\Hom}{\operatorname{Hom}}

\newcommand{\Ind}{\operatorname{Ind}}

\newcommand{\Map}{\operatorname{Map}}

\newcommand{\Mat}{\operatorname{Mat}}
\newcommand{\CAlg}{\operatorname{CAlg}}

\newcommand{\Mod}{\operatorname{Mod}}
\newcommand{\LMod}{\operatorname{LMod}}

\newcommand{\PShv}{\operatorname{PShv}}
\newcommand{\Perf}{\operatorname{Perf}}
\newcommand{\Spec}{\operatorname{Spec}}

\newcommand{\cofib}{\operatorname{cofib}}

\newcommand{\id}{\operatorname{id}}
\newcommand{\map}{\operatorname{map}}

\newcommand{\X}{\mathord{-}}

\newcommand{\cat}[1]{\mathcal{#1}}

\newcommand{\idl}[1]{\mathfrak{#1}}
\newcommand{\Cls}[1]{\mathscr{#1}}
\newcommand{\Cat}[1]{\mathsf{#1}}

\title{Rosenberg's conjecture for the first negative \texorpdfstring{\(K\)}{K}-group}
\author{Ko Aoki}
\address{Max Planck Institute for Mathematics,
  Vivatsgasse 7, 53111 Bonn, Germany
}
\email{aoki@mpim-bonn.mpg.de}
\date{\today}

\begin{document}

\begin{abstract}
  Based on his claims in~1990,
  Rosenberg conjectured in~1997
  that the negative algebraic \(K\)-groups of C*-algebras
  are invariant under continuous homotopy.
  Contrary to his expectation,
  we prove that such invariance holds for~\(K_{-1}\)
  of arbitrary Banach rings
  by establishing a certain continuity result.
  We also construct examples demonstrating
  that similar continuity results do not hold for lower \(K\)-groups.
\end{abstract}

\maketitle

\section{Introduction}\label{s:intro}

We write~\(K\) for
(nonconnective) algebraic \(K\)-theory.
For C*-algebras,
we also have topological \(K\)-theory~\(K^{\top}\).
One key property that distinguishes~\(K^{\top}\)
from~\(K\) is homotopy invariance;
i.e., the tautological map
\(K_{*}^{\top}(A)\to K_{*}^{\top}(\Cls{C}([0,1];A))\)
is an equivalence,
where \(\Cls{C}\) denotes
the ring of continuous function.
This is not true for~\(K\) as
seen by considering the simple case \(A=\CC\) and \({*}=1\).
Nonetheless,
Rosenberg~\cite[Conjecture~2.2]{Rosenberg97}
(cf.~\cite[page~464]{Rosenberg90})
proposed the following:

\begin{conjecture}[Rosenberg]\label{rosenberg}
  For a (real) C*-algebra~\(A\),
  the tautological map
  \begin{equation*}
    K_{*}(A)\to K_{*}(\Cls{C}([0,1];A))
  \end{equation*}
  is an isomorphism
  for \({*}\leq0\).
\end{conjecture}

Suppose that \(A\) is complex
for simplicity.\footnote{The real commutative
  and stable cases
  were proven later in~\cite{k-ros-1}
  and~\cite{KaroubiWodzicki13},
  respectively.
}
This conjecture is known to hold for \({*}\leq0\)
in two extreme cases:
When \(A\) is commutative,
it was proven by Cortiñas--Thom~\cite{CortinasThom12}.
When \(A\) is stable,
this follows 
from the result of
Suslin--Wodzicki~\cite{SuslinWodzicki92}.
For arbitrary~\(A\),
only the case for \(K_{0}\simeq K^{\top}_{0}\) was known.
In this paper,
we prove the first nontrivial result
for any~\(A\).
In fact,
our setting is more general
than C*-algebras:

\begin{Theorem}\label{hinv}
  Let \(A\) be a Banach ring\footnote{Banach rings are assumed to be unital in this paper,
    but the nonunital case follows from the unital case.
  }.
  Then the tautological map
  \begin{equation*}
    K_{-1}(A)\to K_{-1}(\Cls{C}([0,1];A))
  \end{equation*}
  is an isomorphism.
\end{Theorem}

\begin{remark}\label{xwrrxv}
  Our generality of \cref{hinv}
  contradicts Rosenberg's comment
  on~\cite[page~85]{Rosenberg97},
  saying that \cref{rosenberg} should be special to C*-algebras
  rather than general real Banach algebras.
  This view stems from his original proof strategy
  (cf.~\cite[page~464]{Rosenberg90}),
  which relies on Bass delooping.
  In contrast,
  we employ a different delooping method in this paper;
  see \cref{s:deloop}.
\end{remark}

\begin{remark}\label{xbw80p}
The statement of \cref{hinv} is trivial
  when \(A\) is ultranormed,
  since the map \(A\to\Cls{C}([0,1];A)\)
  itself is an isomorphism in that case.
  However,
  its significance is not limited
  to real Banach algebras
  since
  it also applies to real quasi-Banach algebras.
\end{remark}

We deduce \cref{hinv}
from the following ``semicontinuity'' result:

\begin{Theorem}\label{main}
  Let \(A\)
  be a filtered colimit of Banach rings~\(A_{i}\),
  i.e., the completion of the algebraic filtered colimit.
  Then the map
  \begin{equation*}
    \injlim_{i}K_{-1}(A_{i})
    \to
    K_{-1}(A)
  \end{equation*}
  is injective.
\end{Theorem}

\begin{remark}\label{xui98l}
  Since \(\aleph_{1}\)-filtered colimits
  of Banach rings
  can be computed without completion,
  \cref{main} is essentially about sequential colimits.
  For the same reason,
  the reader should feel free to only consider
  sequences instead of general filtered diagrams in this paper.
\end{remark}

Note that this gives a piece
of evidence for the author's conjecture in~\cite{k-ros-2}:

\begin{corollary}\label{xrvu9v}
  The discreteness conjecture~\cite[Conjecture~7.1]{k-ros-2}
  holds for~\(K_{-1}\).
\end{corollary}

Our method is similar
to Drinfeld's study
of \(K_{-1}\)
in~\cite[Section~2]{Drinfeld06}:
We prove that
\(K_{-1}(A)\)
is isomorphic to~\(K_{0}\)
of a certain Calkin-like algebra;
see \cref{xxg7qb} for the precise statement.
This time
we use something closer to
the actual Calkin algebra,
but this construction \emph{cannot} be made Banach:
If such a construction exists,
we get a similar result for arbitrary negative \(K\)-groups
by iterating the same process.
However,
we demonstrate that
\cref{main} is optimal
even in the commutative complex case:

\begin{Theorem}\label{ce}
  There is a sequence
  of commutative complex Banach algebras
  \((A_{n})_{n}\) with colimit~\(A\)
  such that
  the map
  \(\injlim_{n}K_{*}(A_{n})\to K_{*}(A)\)
  is not surjective for any \({*}\leq-1\)
  and not injective for any \({*}\leq-2\).
\end{Theorem}

\begin{remark}\label{xlg9lq}
  This of course does \emph{not} dispute \cref{rosenberg}.
  In some sense, our constructions in \cref{s:ce}
  provide a counterexample to the ``holomorphic version''
  of the discreteness conjecture~\cite[Conjecture~7.1]{k-ros-2}
  for negative \(K\)-groups.
\end{remark}

\subsection*{Organization}\label{ss:outline}

In \cref{s:ban},
we recall the notion of Banach rings.
In \cref{s:zero},
we establish the basic properties of~\(K_{0}\) for them.

In \cref{s:deloop}, we explain our main idea,
which allows us to analyze~\(K_{-1}\) via~\(K_{0}\).
In \cref{s:one}, we prove \cref{main} and derive \cref{hinv} from it.

In \cref{s:ce}, we construct a counterexample for \cref{ce}.
Our construction uses certain rings of holomorphic functions,
which we recall in \cref{s:hol}.

\subsection*{Acknowledgments}\label{ss:ack}

I thank Peter Scholze for useful discussions
and helpful comments on a draft.
I also thank the Max Planck Institute for Mathematics
for its financial support.

\section{Recollection: Banach rings}\label{s:ban}

\begin{definition}\label{1f285b20e2}
  A \emph{norm} on an abelian group~\(M\)
  is a function \(\lVert\X\rVert\colon M\to[0,\infty)\)
  satisfying the following:
  \begin{itemize}
    \item
      \(\lVert x\rVert=0\)
      if and only if \(x=0\)
      for \(x\in M\),
    \item
      \(\lVert -x\rVert=\lVert x\rVert\) for \(x\in M\), and
    \item
      \(\lVert x+y\rVert\leq\lVert x\rVert+\lVert y\rVert\)
      for~\(x\) and~\(y\in M\).
  \end{itemize}
\end{definition}

\begin{definition}\label{x60d6r}
  A \emph{Banach abelian group}
  is a complete normed abelian group.
  A morphism between them
  is a short map,
  i.e., a map of abelian groups with norm \(\leq1\).
  We write~\(\Cat{Ban}\) for this category.
\end{definition}

\begin{remark}\label{x6r7do}
  The category~\(\Cat{Ban}\) appeared
  under the name \(\Cat{Ban}_{\ZZ}^{A,\leq1}\)
  in~\cite[Section~3.1]{BambozziBenBassat16}.
  The focus there was primarily
  on \(\Cat{Ban}_{\ZZ}^{A}\),
  which is the version allowing all bounded maps,
  as it is additive
  whereas \(\Cat{Ban}\) is not.
\end{remark}

\begin{example}\label{xrur6l}
  For \(r>0\),
  the function \(\lVert\X\rVert=r\lvert\X\rvert\)
  is a norm on~\(\ZZ\).
  We write~\(\ZZ_{r}\) for this Banach abelian group.
\end{example}

Since we do not assume
\(\lVert nx\rVert=\lvert n\rvert\lVert x\rVert\)
for \(n\in\ZZ\),
we can consider the following example:

\begin{example}\label{qban}
  Recall that a \emph{quasi-Banach space}
  is a complete real topological vector space
  whose topology is induced by
  a quasinorm \(\lVert\X\rVert'\colon V\to\RR_{\geq0}\);
  i.e., there exists \(C>0\)
  satisfying the following:
  \begin{itemize}
    \item
      \(\lVert x\rVert'=0\) if and only if \(x=0\) for \(x\in V\),
    \item
      \(\lVert ax\rVert'=\lvert a\rvert\lVert x\rVert'\)
      for \(a\in\RR\) and \(x\in V\), and
    \item
      \(\lVert x+y\rVert'\leq C(\lVert x\rVert'+\lVert y\rVert')\).
  \end{itemize}
  The theorem of Aoki~\cite{Aoki42}
  and Rolewicz~\cite{Rolewicz57}
  says that there is a norm
  \(\lVert\X\rVert\colon V\to\RR_{\geq0}\)
  inducing the same topology
  such that
  \(\lVert ax\rVert=\lvert a\rvert^p\lVert x\rVert\)
  holds for \(a\in\RR\) and \(x\in V\),
  where \(p\) satisfies \(2^{(1-p)/p}=C\).
\end{example}
We prove the following categorical property:

\begin{theorem}\label{xdadyx}
  The category~\(\Cat{Ban}\) has a symmetric monoidal structure
  that is objectwise
  given by the (completed) projective tensor product.
  With this symmetric monoidal structure,
  it becomes an object of
  \(\CAlg(\Cat{Pr}^{\aleph_{1}})\);
  i.e.,
  \(\Cat{Ban}\) is \(\aleph_{1}\)-compactly generated and
  the tensor product operations
  commute with colimits in each variable
  and
  preserve \(\aleph_{1}\)-compact objects.
\end{theorem}

\begin{remark}\label{xx58sw}
It is important to consider short maps here;
  the category of Banach abelian groups
  and bounded additive maps (cf.~\cref{x6r7do})
  does neither have products nor coproducts.
\end{remark}

\begin{proof}
  First, we see that the category \(\Cat{Ban}\) has all colimits.
  Coproducts in \(\Cat{Ban}\) can be computed
  as the completions of the \(\ell^{1}\)-norm
  on the algebraic coproducts.
  Coequalizers are obtained by completing the quotient norm.

  To show that \(\Cat{Ban}\) is \(\aleph_{1}\)-compactly generated,
  we use~\cite[Theorem~1.20]{AdamekRosicky94}.
  We wish to show that the set \(\{\ZZ_{r}\mid r>0\}\) (see \cref{xrur6l})
  forms a strong generator.
Consider a monomorphism \(M\hookrightarrow N\) that is not an isomorphism.
  There exists an element \(y\in N\) that is not in the image of~\(M\),
  or its unique preimage \(x\in M\) satisfies \(\lVert x\rVert<\lVert y\rVert\).
  In either case, the map \(\ZZ_{\lVert y\rVert}\to N\)
  sending~\(1\) to~\(y\) does not factor through~\(M\).

  Next, we address the symmetric monoidal structure.
  For any \(M_{1}\), \dots,~\(M_{n}\), and \(N\in\Cat{Ban}\),
  we define a multilinear short map \(f\colon M_{1}\times\cdots\times M_{n}\to N\)
  to be a multilinear map satisfying
  \(\lVert f(x_{1},\ldots,x_{n})\rVert\leq\lVert x_{1}\rVert\cdots\lVert x_{n}\rVert\).
  This definition yields a symmetric colored operad
  (aka symmetric multicategory) underlying~\(\Cat{Ban}\).
  By the universal property of the projective tensor product,
  this structure defines the desired symmetric monoidal structure on~\(\Cat{Ban}\). 

Finally,
  since the binary projective tensor product functor \(\Cat{Ban}\times\Cat{Ban}\to\Cat{Ban}\)
  preserves colimits in each variable,
  it defines an object of \(\CAlg(\Cat{Pr})\).
  To promote it to an object of \(\CAlg(\Cat{Pr}^{\aleph_{1}})\),
  we note that the objects~\(\ZZ_{1}\)
  and \(\ZZ_{r}\otimes\ZZ_{s}\simeq\ZZ_{rs}\) are \(\aleph_{1}\)-compact.
\end{proof}

\begin{corollary}\label{x7bc88}
  The symmetric monoidal category~\(\Cat{Ban}\)
  has internal mapping objects.
  Concretely, \(\Hom(M,N)\)\footnote{We write \(\Map(M,N)\)
    for the set of maps
    to avoid confusion.
    We have \(\Map(M,N)=\Hom(M,N)_{\leq1}\).
  } is given as
  the abelian groups of
  bounded linear functions
  with the norm given
  by \(\sup_{x\neq0}\lVert\X(x)\rVert/\lVert x\rVert\).
\end{corollary}

\begin{definition}\label{xs4ru9}
  We call an object
  of \(\Alg(\Cat{Ban})\)
  a \emph{Banach ring}.
  Concretely,
  it is a Banach abelian group~\(A\)
  equipped with a ring structure
  satisfying the following:
  \begin{itemize}
    \item
      \(\lVert1\rVert\leq1\) and
    \item
      \(\lVert xy\rVert\leq\lVert x\rVert\lVert y\rVert\)
      for~\(x\) and~\(y\in A\).
  \end{itemize}
\end{definition}

\begin{remark}\label{xma7bp}
  Berkovich studied a similar notion
  in~\cite[Section~1.1]{Berkovich90}.
  One difference
  is that the zero ring is a Banach ring in our sense,
  whereas it is not in his sense since he required \(\lVert1\rVert=1\).
\end{remark}

The point here is that
the category of Banach rings
is abstractly defined.
For example,
for a Banach ring~\(A\),
the category of Banach \(A\)-modules
makes sense without defining
objects and morphisms
in an ad~hoc way;
it is just \(\LMod_{A}(\Cat{Ban})\).
We also automatically get the definition of \(\Hom_{A}(M,N)\)
for Banach \(A\)-modules~\(M\) and~\(N\).

\begin{example}\label{l1}
  Let \(A\) be a Banach ring.
  We write \(\ell^{1}(A)\)
  for an algebraic submodule of \(A^{\NN}\)
  consisting of sequences
  with finite \(\ell^{1}\)-norm.
  This has a Banach \(A\)-module structure
  and is the countable coproduct of~\(A\)
  in \(\LMod_{A}(\Cat{Ban})\).
  The countable product is
  given by \(\ell^{\infty}(A)\).
\end{example}

\begin{example}\label{real_ban}
  The real numbers~\(\RR\)
  with the usual norm
  is a Banach ring.
  The category \(\Mod_{\RR}(\Cat{Ban})\)
  is the category of real Banach spaces:
  A~priori we only have \(\lVert ax\rVert\leq\lvert a\rvert\lVert x\rVert\)
  for \(a\in\RR^{\times}\),
  but since we also have \(\lVert x\rVert\leq\lvert 1/a\rvert\lVert ax\rVert\),
  it must be an equality.
  The category \(\Alg_{\RR}(\Cat{Ban})\)
  is the category of real Banach algebras.

  The same argument shows that
  when \(F\) is a Banach ring
  that is a field with a multiplicative norm~\(\lvert\X\rvert\),
  we have \(\lVert ax\rVert=\lvert a\rvert\lVert x\rVert\)
  in any Banach \(F\)-module.
\end{example}

\begin{example}\label{xf0apa}
  Let \(A\) be a Banach ring.
  For any compact Hausdorff space~\(X\),
  the ring of continuous functions
  \(\Cls{C}(X;A)\)
  becomes a Banach ring with the supremum norm.
\end{example}

\section{\texorpdfstring{\(K_{0}\)}{K0} of Banach rings}\label{s:zero}

We prove the following two properties
of~\(K_{0}\) special to Banach rings:

\begin{theorem}\label{x90xl4}
  Let \(A\) be
  a filtered colimit of Banach rings~\(A_{i}\).
  Then the map
  \begin{equation*}
    \injlim_{i}K_{0}(A_{i})\to K_{0}(A)
  \end{equation*}
  is an isomorphism.
\end{theorem}

\begin{theorem}\label{xihxwg}
  Let \(A\) be a Banach ring
  and \(I\subset J\subset A\) be ideals
  of (the underlying ring of)~\(A\) having the same closure.
  Then the map
  \(K_{0}(A/I)\to K_{0}(A/J)\) is injective.\footnote{The proof also shows that
    \(K_{1}(A/I)\to K_{1}(A/J)\) is surjective,
    which we do not use in this paper.
  }
\end{theorem}

By combining \cref{x90xl4,xihxwg},
we see the following:

\begin{corollary}\label{xlcvuz}
  Let \(A=\injlim_{i}A_{i}\) be
  a filtered colimit of Banach rings
  and \(I_{i}\subset A_{i}\) and \(I\subset A\) be compatible ideals.
  Suppose that
  the union of the images of~\(I_{i}\) is dense in~\(I\).
  Then the map
  \begin{equation*}
    \injlim_{i}K_{0}(A_{i}/I_{i})
    \to
    K_{0}(A/I)
  \end{equation*}
  is injective.
\end{corollary}

\begin{proof}
  By \cref{x90xl4},
  the map
  \begin{equation*}
    \injlim_{i}K_{0}(A_{i}/J_{i})
    \to
    K_{0}(A/J)
  \end{equation*}
  is an isomorphism,
  where \(J_{i}\) and~\(J\)
  are the closures of~\(I_{i}\) and~\(I\), respectively.
  Hence the desired result follows from \cref{xihxwg}.
\end{proof}

Our proof methods are similar to the real case;
we need to explicitly check that
functional calculus works in this general situation.

For a ring~\(A\), we write~\(E(A)\) for
the quotient of the set of idempotents
in~\(A\) by the equivalence relation
given by conjugation;
i.e.,
\(e\sim f\)
if and only if there exists \(u\in A^{\times}\)
satisfying \(eu=uf\).
Then \(K_{0}(A)\)
is identified with
the group completion
of \(\injlim_{n}E(\Mat_{n}(A))\).
Therefore,
for a ring map \(A\to B\),
if the map \(E(\Mat_{n}(A))\to E(\Mat_{n}(B))\)
is an isomorphism
for any~\(n\),
so is \(K_{0}(A)\to K_{0}(B)\).

\begin{lemma}\label{x1s72v}
  Let \(u\in A\) be an element in a Banach ring
  satisfying \(\lVert u-1\rVert<1\).
  Then \(u\) is invertible.
\end{lemma}

\begin{proof}
  The series
  \(\sum_{n=0}^{\infty}(1-u)^{n}\)
  converges and is the inverse.
\end{proof}

\begin{lemma}\label{x91cml}
  For an invertible element~\(u\in A^{\times}\)
  in a Banach ring,
  there is \(\epsilon>0\)
  such that any element~\(v\)
  satisfying \(\lVert v-u\rVert<\epsilon\)
  is invertible.
\end{lemma}

\begin{proof}
  This follows from \cref{x1s72v}.
\end{proof}

\begin{lemma}\label{x7f05x}
  For a Banach ring~\(A\),
  the function \((\X)^{-1}\colon A^{\times}\to A^{\times}\)
  is continuous.
\end{lemma}

\begin{proof}
  It suffices to prove the continuity at~\(1\),
  which follows from the proof of \cref{x1s72v}.
\end{proof}

\begin{lemma}\label{app_conj}
  Let \(A\) be a Banach ring.
  For any idempotent~\(e\),
  there is \(\delta>0\)
  such that
  \([e]=[f]\) in \(E(A)\)
  for any idempotent~\(f\)
  satisfying \(\lVert e-f\rVert<\delta\).
\end{lemma}

\begin{proof}
  For idempotents~\(e\) and~\(f\),
  we consider
  \begin{equation*}
    u
    =ef+(1-e)(1-f)
    =1-e(e-f)+(e-f)e-(e-f)^{2}.
  \end{equation*}
  Then we see that \(eu=ef=uf\).
  Since
  \(\lVert u-1\rVert\leq2\lVert e\rVert\lVert e-f\rVert+\lVert e-f\rVert^{2}\),
  by \cref{x1s72v},
  it is invertible when \(\lVert e-f\rVert\) is small enough.
\end{proof}

\begin{lemma}\label{app_idem}
  Let \(a\in A\) be an element
  satisfying \(\lVert a^{2}-a\rVert<1/4\).
  Then we have an idempotent~\(e\)
  satisfying \(ea=ae\) and \(\lVert e-a\rVert\leq h(\lVert a^{2}-a\rVert)\),
where \(h\colon[0,1/4)\to[0,\infty)\) is given by
  \begin{equation*}
    h(t)
    =
    \frac{1-\sqrt{1-4t}}2.
  \end{equation*}
\end{lemma}

\begin{proof}
We consider the element
\begin{equation*}
    x
    =
    -\sum_{n=1}^{\infty}
    2^{2n-1}\binom{1/2}{n}(a^{2}-a)^{n},
  \end{equation*}
  which converges by our assumption on \(a^{2}-a\);
  note that \(2^{2n-1}\binom{1/2}{n}\) is an integer.
  Then \(e=a+x\) is the desired idempotent.
\end{proof}

We prove the following strengthening of \cref{x90xl4}:

\begin{proposition}\label{xp0bgr}
  Let \(A\) be
  a filtered colimit of Banach rings~\((A_{i})_{i}\).
  Then the map
  \(\injlim_{i}E(A_{i})\to E(A)\)
  is an isomorphism.
\end{proposition}

\begin{proof}
We write \(g_{i}\colon A_{i}\to A\)
  for the tautological map.

  We first show the injectivity.
  Suppose that idempotents~\(e_{i}\) and \(f_{i}\in A_{i}\)
  with images \(e\) and~\(f\) in~\(A\)
  satisfy \(eu=uf\)
  for an element \(u\in A^{\times}\).
  Consider a small enough positive real \(\epsilon>0\).
  By \cref{x91cml,x7f05x},
  by advancing~\(i\),
  we can assume that
  there is an element \(u_{i}\in A_{i}^{\times}\)
  such that \(\lVert g_{i}(u_{i})-u\rVert<\epsilon\) and
  \(\lVert g_{i}(u_{i}^{-1})-u^{-1}\rVert<\epsilon\) hold.
  Since
\begin{equation*}
    \lVert g_{i}(u_{i}^{-1}e_{i}u_{i}-f_{i})\rVert
    \leq
    \epsilon
    \lVert e\rVert
    (\epsilon+\lVert u\rVert+\lVert u^{-1}\rVert),
  \end{equation*}
  by sufficiently advancing~\(i\),
  we can assume that
  \(\lVert u_{i}^{-1}e_{i}u_{i}-f_{i}\rVert<\epsilon\).
  By \cref{app_conj},
  we see that \(e_{i}\) and~\(f_{i}\) are equivalent.

  Then we show the surjectivity.
  Let \(e\in A\) be an idempotent.
  Consider a small enough positive real \(\epsilon>0\).
  We can choose~\(i\) and \(a_{i}\in A_{i}\)
  satisfying
  \(\lVert a_{i}^{2}-a_{i}\rVert<\epsilon\)
  and
  \(\lVert g_{i}(a_{i})-e\rVert<\epsilon\).
  Then by \cref{app_idem},
  we obtain an idempotent \(e_{i}\in A_{i}\)
  satisfying \(\lVert e_{i}-a_{i}\rVert\leq h(\epsilon)\).
  In this case,
  \begin{equation*}
    \lVert g_{i}(e_{i})-e\rVert
    \leq
    \lVert e_{i}-a_{i}\rVert
    +
    \lVert g_{i}(a_{i})-e\rVert
    <
    h(\epsilon)+\epsilon
  \end{equation*}
  and therefore, by \cref{app_conj},
  the map \(E(A_{i})\to E(A)\)
  maps~\([e_{i}]\) to~\([e]\).
\end{proof}

We then prove the other claim:

\begin{proof}[Proof of \cref{xihxwg}]
It suffices to prove that
  \(J/I\) is a radical nonunital ring
  since such a nonunital ring has vanishing~\(K_{0}\).
  Consider an element \(a\in J\).
  There is an element \(a'\in I\)
  such that \(\lVert a'-a\rVert<1\).
  By \cref{x1s72v},
  \(1+a-a'\) is invertible in~\(A\).
  Hence \(1+a\) is invertible in \(A/I\).
\end{proof}

\section{Delooping Banach rings}\label{s:deloop}

For a Banach ring~\(A\),
we identify \(K_{-1}(A)\) with \(K_{0}\)
of a certain algebraic quotient of a Banach ring:
We first realize~\(A\) as a corner\footnote{For a pair of a ring~\(B\) and an idempotent~\(f\),
  its \emph{corner}
  is the ring \(fBf\) (with unit~\(f\)),
  which is identified with \(\End_{B}(Bf)\).
} of
a big Banach ring~\(\End_{A}(\ell^{1}(A))
=\Hom_{A}(\ell^{1}(A),\ell^{1}(A))\).
We then show
that
this big ring has vanishing \(K\)-theory
and
how \(K\) behaves with respect to corners:

\begin{theorem}\label{isr}
  For a Banach ring~\(A\),
  any additive invariant\footnote{Here we do not require an additive invariant
    to preserve filtered colimits.
  } maps
  \(\End_{A}(\ell^{1}(A))\) to zero.
\end{theorem}

\begin{theorem}\label{idem}
  Consider an idempotent~\(e\)
  in an associative ring~\(A\).
  For any localizing invariant\footnote{Here we do not require a localizing invariant
    to preserve filtered colimits.
  }~\(F\),
  there is a canonical fiber sequence
  \(F(eAe)\to F(A)\to F(A/AeA)\).
\end{theorem}

\begin{corollary}\label{xxg7qb}
  For a Banach ring~\(A\),
  consider the direct summand \(A\subset\ell^{1}(A)\)
  corresponding to the zeroth entry.
  We have a canonical isomorphism
  \begin{equation*}
    K_{-1}(A)
    \simeq
    K_{0}\biggl(
      \frac
      {\Hom_A(\ell^1(A),\ell^1(A))}
      {\Hom_A(\ell^1(A),A)\otimes_A\ell^1(A)}
    \biggr),
  \end{equation*}
  where
  \(\otimes_{A}\) denotes the \emph{algebraic} (i.e., uncompleted) tensor product.
\end{corollary}

\begin{proof}
  The composite
  \(\ell^{1}(A)\to A\to\ell^{1}(A)\)
  determines an element
  in \(\End_{A}(\ell^{1}(A))\),
  which is an idempotent.
  The desired result follows from \cref{isr,idem}.
\end{proof}

\begin{proof}[Proof of \cref{isr}]
We consider the full subcategory~\(\cat{A}\)
  of \(\Mod_{A}(\Cat{Ban})\)
  spanned by at most countable coproducts of~\(A\).
  By considering the internal mapping object,
  we obtain an \(\Cat{Ab}\)-enriched category~\(\cat{B}\).
Then \(\cat{B}\) is additive
  and hence we can forget the enrichment.
  We write~\(\cat{C}\)
  for the \(\infty\)-category of compact objects of
  the stabilization of
  \(\PShv_{\Sigma}(\cat{B})\).
  By definition, \(\cat{C}\) is generated by~\(\ell^{1}(A)\)
  as an idempotent-complete stable \(\infty\)-category,
  and therefore is equivalent to \(\Perf(\End_{A}(\ell^{1}(A)))\).
  Therefore,
  it suffices to show that
  \(F(\cat{C})\) vanishes
  for any additive invariant~\(F\).

  We consider the endofunctor on \(\Mod_{A}(\Cat{Ban})\)
  that maps~\(M\) to the countable coproduct of~\(M\).
  This restricts to~\(\cat{A}\)
  and induces endofunctors on~\(\cat{B}\)
  and therefore on~\(\cat{C}\),
  for which we write~\(T\).
  Since we have \({\id}\oplus T\simeq T\),
  the identity morphism of \(F(\cat{C})\) must be zero.
\end{proof}

\begin{remark}\label{x2ywou}
  To prove \cref{isr},
  we can also argue
  more concretely as follows:
  We first equip \(\End_{A}(\ell^{1}(A))\)
  with an \emph{infinite-sum ring} structure
  in the sense of~\cite{FarrellWagoner72}.
  We then observe that
  additive invariants
  carry any infinite-sum ring to zero
  (for~\(K\), this part was proven by Wagoner~\cite{Wagoner72}
  using a concrete argument).
\end{remark}

\begin{proof}[Proof of \cref{idem}]
  We write \(\cat{A}\) for \(\Perf(A)\).
  Let \(\cat{A}'\) be
  the full subcategory of~\(\cat{A}\)
  generated by \(Ae\)
  as an idempotent-complete stable \(\infty\)-category.
  Let \(\cat{A}''\) be the Verdier quotient.
  Since \(\End_{A}(Ae)\) is the corner,
  \(\cat{A}'\) is equivalent to \(\Perf(eAe)\).
  Hence it suffices to identify \(\cat{A}''\)
  with \(\Perf(A/AeA)\).
  To compute~\(\cat{A}''\),
  we \(\Ind\)-extend the situation to get the diagram
  \begin{equation*}
    \begin{tikzcd}
      \Ind(\cat{A}')\ar[r,shift left,"j_{!}"]&
      \Ind(\cat{A})\ar[r,shift left,"i^{*}"]\ar[l,shift left,"j^{*}"]&
      \Ind(\cat{A}'')\rlap.\ar[l,shift left,"i_{*}"]
    \end{tikzcd}
  \end{equation*}
  Note that \(i^{*}A\) generates~\(\cat{A}''\)
  and its endomorphism \(\E_{1}\)-ring can be computed as
  \begin{equation*}
    \map_{\cat{A}'}(i^{*}A,i^{*}A)
    \simeq
    \map_{\cat{A}}(A,i_{*}i^{*}A)
    \simeq
    \cofib(
    \map_{\cat{A}}(A,j_{!}j^{*}A)
    \to
    \map_{\cat{A}}(A,A)
    ),
  \end{equation*}
  where \(\map\) denotes the mapping spectrum.
  Therefore,
  the desired result follows from \cref{xy7dg1} below.
\end{proof}

\begin{lemma}\label{xy7dg1}
  For an idempotent~\(e\) in a ring~\(A\),
  we write
  \(\cat{C}\subset\LMod(A)\)
  for the full subcategory generated by~\(Ae\)
  under colimits and shifts.
  Then its coreflector is given as
  \(AeA\otimes_{A}\X\).
\end{lemma}

\begin{proof}
  By the compactness of~\(Ae\),
  the coreflector preserves colimits.
  Therefore,
  it is reduced to showing that
  the map
  \begin{equation*}
    \map_{A}(Ae,AeA)
    \to
    \map_{A}(Ae,A)
  \end{equation*}
  induced by \(AeA\hookrightarrow A\) is an equivalence.
  By realizing this morphism as a direct summand of
  \(AeA\hookrightarrow A\),
  we can write it as
  \begin{equation*}
    \{aea'\in AeA\mid(1-e)aea'=0\}
    \hookrightarrow
    \{a\in A\mid(1-e)a=0\}.
  \end{equation*}
  The desired result follows from the observation
  that \((1-e)a=0\)
  implies \(a=1ea\).
\end{proof}

\section{\texorpdfstring{\(K_{-1}\)}{K-1} of Banach rings}\label{s:one}

We prove \cref{main}
by studying
how the construction in \cref{s:deloop}
behaves with respect to filtered colimits.
We then deduce \cref{hinv} from it.
We start with some elementary observations:

\begin{lemma}\label{xp2msm}
  Let \(e\) be an idempotent in a Banach ring~\(A\).
  The corner \(eAe\subset A\)
  is a Banach ring with the induced norm
  if and only if \(\lVert e\rVert\leq1\).
\end{lemma}

\begin{proof}
By \(e^{2}=e\),
  any idempotent satisfies 
  \(\lVert e\rVert=0\)
  or \(\lVert e\rVert\geq1\).

  When \(\lVert e\rVert=0\),
  the claim is trivial.
  When \(\lVert e\rVert=1\),
  we can directly check that
  \(eAe\) is a Banach ring.
  When \(\lVert e\rVert>1\),
  it cannot be a Banach ring
  since \(\lVert1\rVert\leq1\) is required.
\end{proof}

\begin{lemma}\label{x82cbs}
  Let \(A=\injlim_{i}A_{i}\)
  be a filtered colimit of Banach rings.
  Let \(e_{i}\in A_{i}\)
  be a compatible family of idempotents
  satisfying \(\lVert e_{i}\rVert\leq1\),
  determining an idempotent~\(e\in A\).
  \begin{enumerate}
    \item\label{i:eae}
      The map
      \(\injlim_{i}e_{i}A_{i}e_{i}
      \to
      eAe\)
      where the colimit is taken in \(\Alg(\Cat{Ban})\)
      is an isomorphism.
    \item\label{i:aea}
      The union of the image of
      \(A_{i}e_{i}A_{i}\)
      and \(AeA\)
      have the same closure in~\(A\).
  \end{enumerate}
\end{lemma}

\begin{proof}
We first prove~\cref{i:eae}.
  By \cref{xp2msm},
  we know that the colimit
  is the closure of the union
  of the image of \(e_{i}A_{i}e_{i}\),
  which we wish to identify with~\(eAe\).
  For \(a\in A\),
  we take a sequence \((a_{n})_{n}\) in the image of the union
  of the images of \(A_{i}\) converging to~\(a\).
  Then \((ea_{n}e)_{n}\) is a sequence
  in the image of the union of the images of \(e_{i}A_{i}e_{i}\)
  converging to~\(eae\).

  We then prove~\cref{i:aea}.
  For \(a\) and \(b\in A\),
  we approximate them similarly by \((a_{n})_{n}\) and~\((b_{n})_{n}\)
  which are in the union of the images of~\(A_{i}\).
  Then \((a_{n}eb_{n})_{n}\) converges to~\(aeb\).
\end{proof}

\begin{proof}[Proof of \cref{main}]
  We consider \(B_{i}=\End_{A_{i}}(\ell^{1}(A_{i}))\)
  and write~\(B\) for its colimit in \(\Alg(\Cat{Ban})\).
  Let \(e_{i}\in B_{i}\) be the idempotent
  corresponding to the zeroth entry.
  Then this family is compatible
  and determines an idempotent \(e\in B\).
  By~\cref{i:eae} of \cref{x82cbs},
  \(eBe\) is isomorphic to~\(A\).
  Therefore,
  by \cref{idem},
  we obtain the diagram
  \begin{equation*}
    \begin{tikzcd}
      \injlim_{i}K_{0}(B_{i})\ar[r]\ar[d]&
      \injlim_{i}K_{0}(B_{i}/B_{i}e_{i}B_{i})\ar[r]\ar[d]&
      \injlim_{i}K_{-1}(A_{i})\ar[r]\ar[d]&
      \injlim_{i}K_{-1}(B_{i})\ar[d]\\
      K_{0}(B)\ar[r]&
      K_{0}(B/BeB)\ar[r]&
      K_{-1}(A)\ar[r]&
      K_{-1}(B)
    \end{tikzcd}
  \end{equation*}
  with exact rows.
  Since \(K(B_{i})\) and~\(K(B)\) vanish
  by \cref{isr},
  the third vertical arrow is isomorphic to the second,
  which is injective by~\cref{i:aea} of \cref{x82cbs}
  and \cref{xlcvuz}.
\end{proof}

We then deduce the homotopy invariance property of \(K_{-1}\):

\begin{proof}[Proof of \cref{hinv}]
  This argument follows the proof of~\cite[Lemma~7.28]{k-ros-2},
  but note that here we do not use full excision
  and hence we do not require the Dugundji extension theorem.

  For real numbers \(a\leq b\),
  we write \(M(a,b)\)
  for the cokernel of
  the split injection
  \(K_{-1}(A)\to K_{-1}(\Cls{C}([a,b];A))\).
  We have to show that \(M(0,1)\) vanishes.
  We assume that
  there is a nonzero class \(\alpha\in M(0,1)\)
  and wish to derive a contradiction.

  We first see that
  \(M(a,c)\to M(a,b)\oplus M(b,c)\) is an isomorphism
  for \(a\leq b\leq c\).
  We consider the diagram
  \begin{equation*}
    \begin{tikzcd}
      \Cls{C}([a,c];A)\ar[r]\ar[d]&
      \Cls{C}([a,b];A)\ar[d]\\
      \Cls{C}([b,c];A)\ar[r]&
      \Cls{C}(\{b\};A)\rlap,
    \end{tikzcd}
  \end{equation*}
  which consists of split surjections of rings.
  The desired isomorphism
  follows from excision for negative \(K\)-theory.

  From this,
  we see that \(\alpha\)
  is nonzero in either \(M(0,1/2)\) or \(M(1/2,1)\).
  Repeating this process,
  we can pick \(p\in[0,1]\)
  such that
  \(\alpha\) is nonzero in
  \begin{equation*}
    \injlim_{p\in [a,b]\subset[0,1]}
    M(a,b),
  \end{equation*}
  where \([a,b]\) runs over the neighborhoods of \(p\in[0,1]\).
  However, this contradicts \cref{main},
  which implies that this colimit vanishes.
\end{proof}

\begin{remark}\label{x18a95}
  By using \cref{x90xl4} instead of \cref{main}
  in the proof of \cref{hinv} above,
  we obtain the homotopy invariance property
  of~\(K_{0}\).
\end{remark}

\section{Recollection: rings of holomorphic functions}\label{s:hol}

From now on,
all rings are assumed to be commutative
and all Banach algebras are always complex.

This section is a preparation for the proof of \cref{ce}
in the next section.
We require
complex-analytic spaces to be reduced and separated.
Recall that a \emph{Stein compact set}
is a locally ringed space over~\(\Spec\CC\) that
can be realized
as a holomorphically convex closed subspace of some Stein space.

\begin{definition}\label{xf8gdf}
  For a Stein compact set~\(Z\),
  we write \(\Cls{O}(Z)\)
  for the ring of global sections.
  We write \(\Cls{A}(Z)\) for its completion
  with respect to the supremum norm.
\end{definition}

\begin{example}\label{xf6r1i}
  In \cref{xf8gdf},
  consider the case where
  \(Z\) is the Stein compact set
  corresponding to the unit disk inside~\(\CC\).
  Then \(\Cls{O}(Z)\) is the ring
  of overconvergent functions
  and \(\Cls{A}(Z)\) is the disk algebra.
\end{example}

\begin{example}\label{stein}
  In \cref{xf8gdf},
  Let \(Y\) be a real-analytic space
  and \(Z\) a compact subset of its real points.
  Then the inclusion \(Z\subset Y_{\CC}\)
  gives~\(Z\)
  the structure of a Stein compact set.
  In this case, \(\Cls{O}(Z)\)
  is the ring of overconvergent real-analytic functions
  and therefore \(\Cls{A}(Z)\) coincides with
  the ring of continuous functions~\(\Cls{C}(Z)\).
\end{example}

\begin{example}\label{xah68l}
Consider Stein compact sets~\(Z\) and~\(Z'\)
  satisfying \(Z'\Subset Z\),
  i.e., there is a Stein space~\(U\)
  such that \(Z'\subset U\subset Z\).
  Then we can fill the diagram
  \begin{equation*}
    \begin{tikzcd}
      \Cls{O}(Z)\ar[r]\ar[d]&
      \Cls{O}(Z')\ar[d]\\
      \Cls{A}(Z)\ar[r]\ar[ur,dashed]&
      \Cls{A}(Z')
    \end{tikzcd}
  \end{equation*}
  by a dashed arrow
  since
  the restriction
  \(\Cls{O}(Z)\to\Cls{O}(U)\)
  factors through the completion.
\end{example}

The following is a key tool for our proof of \cref{ce}:

\begin{example}\label{ind_ban}
  Suppose that a Stein compact set~\(Z\)
  is a holomorphically convex closed subspace
  of a Stein space~\(X\).
  We take its Stein compact neighborhood
  \begin{equation*}
    Z_{0}\Supset Z_{1}\Supset\dotsb
  \end{equation*}
  which converges to~\(Z\).
  By definition,
  \(\Cls{O}(Z)\) is the colimit of \(\Cls{O}(Z_{n})\).
  By \cref{xah68l},
  we see that it is also the colimit of \(\Cls{A}(Z_{n})\).
  This also shows that the colimit of \(\Cls{A}(Z_{n})\)
  in \(\Alg(\Cat{Ban})\)
  (or equivalently, \(\CAlg_{\CC}(\Cat{Ban})\); cf.~\cref{real_ban})
  is \(\Cls{A}(Z)\).
\end{example}

Finally,
we recall a classical fact:

\begin{theorem}\label{o_reg}
  Suppose that \(Z\) is a Stein compact subset of
  a complex manifold.
  When \(Z\) is semianalytic,
  \(\Cls{O}(Z)\) is a regular ring.
\end{theorem}

\begin{proof}
  Frisch's result~\cite{Frisch67}
  states that \(\Cls{O}(Z)\) is noetherian in that case.
  To check that it is regular,
  it suffices to consider the formal completion
  at each maximal ideal.
  By~\cite[Corollary~3.3]{Zame76},
  any maximal ideal corresponds to a point of~\(Z\).
  In that case,
  the completion is just the ring
  of formal power series,
  which is regular.
\end{proof}

\section{Lower \texorpdfstring{\(K\)}{K}-groups of commutative complex Banach algebras}\label{s:ce}

To prove \cref{ce},
we construct the following counterexamples separately:

\begin{theorem}\label{ce_sur}
  For each \({*}\leq-1\),
  there is a sequence of commutative complex Banach algebras
  \((A_{n})_{n}\) with colimit~\(A\)
  such that
  the map
  \(\injlim_{n}K_{*}(A_{n})\to K_{*}(A)\)
  is not surjective.
\end{theorem}

\begin{theorem}\label{ce_inj}
  For each \({*}\leq-2\),
  there is a sequence of commutative complex Banach algebras
  \((A_{n})_{n}\) with colimit~\(A\)
  such that
  the map
  \(\injlim_{n}K_{*}(A_{n})\to K_{*}(A)\)
  is not injective.
\end{theorem}

We can combine them
to obtain the desired counterexample:

\begin{proof}[Proof of \cref{ce}]
  We use the counterexamples from \cref{ce_sur,ce_inj},
  denoting them as~\(B_{n}^{(*)}\) and~\(C_{n}^{(*)}\), respectively.
  Then we consider
  \begin{equation*}
    A_{n}
    =\prod_{{*}\leq-1}B_{n}^{(*)}
    \times\prod_{{*}\leq-2}C_{n}^{(*)},
  \end{equation*}
  where the products are taken in \(\Alg(\Cat{Ban})\)
  (or equivalently, in \(\CAlg_{\CC}(\Cat{Ban})\); cf.~\cref{real_ban}).
  We write \(A\) for the (completed) colimit.
  This is the desired counterexample
  since \((B_{n}^{(*)})_{n}\)
  and \((C_{n}^{(*)})_{n}\)
  are retracts of \((A_{n})_{n}\)
  as sequences of \emph{nonunital} rings.
\end{proof}

We first construct surjectivity counterexamples:

\begin{proof}[Proof of \cref{ce_sur}]
  We fix \(d\geq1\)
  and construct a counterexample for \({*}=-d\).
  By \cref{ind_ban},
  it suffices to find a Stein compact set~\(Z\)
  such that \(K_{-d}(\Cls{O}(Z))\to K_{-d}(\Cls{A}(Z))\)
  is not surjective.

  We use~\cref{stein}.
  We consider~\(Y\)
  to be the subvariety of \(\RR^{d+1}\)
  cut out by the equation \(x_{1}^{2}+\dotsb+x_{d+1}^{2}=1\)
  and take \(Z\) to be all the real points of~\(Y\).
  In this case,
  \(\Cls{O}(Z)\) is regular by \cref{o_reg}
  and thus its \(K_{-d}\) vanishes.
  However,
  \(K_{-d}\) of
  \(\Cls{A}(Z)=\Cls{C}(S^{d})\)
  is isomorphic to~\(\ZZ\)
  by
  Friedlander--Walker's result~\cite[Theorem~5.1]{FriedlanderWalker01C}.
\end{proof}

The main ingredient for injectivity counterexamples
is the following from~\cite{Reid87}:

\begin{theorem}[Reid]\label{reid}
  For a field~\(k\) and \(d\geq2\),
  there is a normal \(d\)-dimensional affine \(k\)-algebra
  with a single singular point
  whose \(K_{-d}\) does not vanish.
\end{theorem}

\begin{proof}[Proof of \cref{ce_inj}]
  We fix \(d\geq2\)
  and construct a counterexample for \({*}=-d\).
  By \cref{ind_ban},
  it suffices to find a Stein compact set~\(Z\)
  such that \(K_{-d}(\Cls{O}(Z))\to K_{-d}(\Cls{A}(Z))\)
  is not injective.

  We fix a counterexample \((A,\idl{m})\)
  of \cref{reid}
  over \(k=\CC\).
  We write \(Z\) for the Stein compact set
  associated to the point~\(\idl{m}\)
  inside the analytification of \(\Spec A\).
  Since \(\Cls{A}(Z)\simeq\CC\),
  it suffices to show that
  \(\Cls{O}(Z)\) has nonvanishing \(K_{-d}\).
  We consider the maps
  \begin{equation*}
    A\to\Cls{O}(Z)\to A_{\idl{m}}^{\wedge}.
  \end{equation*}
By~\cite[Proposition~1.6]{Weibel80},
  the composite induces
  an isomorphism on~\(K_{-d}\).
  Therefore,
  \(K_{-d}(\Cls{O}(Z))\) is a retract
  of \(K_{-d}(A)\), which is nonzero.
\end{proof}

\let\top\oldtop \bibliographystyle{plain}
 \newcommand{\yyyy}[1]{}

\end{document}